\documentclass[12pt,a4paper]{amsart}
\usepackage{amsfonts}
\usepackage{amsthm}
\usepackage{amsmath}
\usepackage{amscd}
\usepackage[latin2]{inputenc}
\usepackage{t1enc}
\usepackage[mathscr]{eucal}
\usepackage{indentfirst}
\usepackage{graphicx}
\usepackage{graphics}
\usepackage{multicol}
\usepackage{subcaption}
\usepackage{pgffor}
\usepackage[margin=3cm]{geometry}

\numberwithin{equation}{section}

\theoremstyle{plain}
\newtheorem{Th}{Theorem}[section]
\newtheorem{Lemma}[Th]{Lemma}
\newtheorem{Cor}[Th]{Corollary}

\newtheorem{Pro}[Th]{Proposition}

\theoremstyle{definition}

\newtheorem{Conj}[Th]{Conjecture}
\newtheorem{Rem}[Th]{Remark}
\newtheorem{?}[Th]{Problem}

\newtheorem{Q}[Th]{Question}

\makeatletter
\newtheorem*{rep@theorem}{\rep@title}
\newcommand{\newreptheorem}[2]{%
\newenvironment{rep#1}[1]{%
 \def\rep@title{#2 \ref{##1}}%
 \begin{rep@theorem}}%
 {\end{rep@theorem}}}
\makeatother

\newreptheorem{Th}{Theorem}
\newreptheorem{Pro}{Proposition}

\DeclareMathSizes{5}{5}{5}{5}

\usepackage{float}
\usepackage{tikz}

\usetikzlibrary{arrows}
\tikzstyle{n}=[circle,draw,thick,scale=0.5]
\tikzstyle{nt}=[circle,draw,thick,fill,scale=0.5]
\tikzstyle{nc}=[circle,draw,thick,fill]

\title{Some coefficient sequences related to the descent polynomial}

\author{Ferenc Bencs} 
 \address{Central European University, Department of Mathematics
 \\ H-1051 Budapest
 \\ Zrinyi u. 14. \\ Hungary \& Alfr\'ed R\'enyi Institute of Mathematics\\ H-1053 Budapest\\ Re\'altanoda u. 13-15.} 
 \email{ferenc.bencs@gmail.com}



\begin{document}
\maketitle

\begin{abstract}
 The descent polynomial of a finite $I\subseteq \mathbb{Z}^+$ is the polynomial $d(I,n)$, for which the evaluation at $n>\max(I)$ is the number of permutations on $n$ elements, such that $I$ is the set of indices where the permutation is descending. In this paper we will prove some conjectures concerning coefficient sequences of $d(I,n)$. As a corollary we will describe some zero-free regions for the descent polynomial.
\end{abstract}

\section{Introduction}

Denote the group of permutations on $[n]=\{1,\dots,n\}$ by $\mathcal{S}_n$ and for a permutation $\pi\in \mathcal{S}_n$, the set of descending position is 
\[
  Des(\pi)=\{i\in [n-1]~|~\pi_i>\pi_{i+1}\}.
\]
We would like to investigate the number of permutations with a fixed descent set. More precisely, for a finite $I\subseteq \mathbb{Z}^+$ let $m=\max(I\cup \{0\})$. Then for $n>m$ we can count the number of permutations with descent set $I$, that we will denote by
\[
  d(I,n)=|D(I,n)|=|\{\pi \in \mathcal{S}_n~|~ Des(\pi)=I\}|.
\]

This function was shown to be a degree $m$ polynomial in $n$ by MacMahon in \cite{macmahon2004}. In order to investigate this polynomial we extend the domain to $\mathbb{C}$, and for this paper we call $d(I,n)$ the descent polynomial of $I$. 

This polynomial was recently studied in the article of Diaz-Lopez, Harris, Insko, Omar and Sagan \cite{Sagan}, where the authors found a new recursion which was motivated by the peak polynomial.
The paper investigated the roots of descent polynomials and their coefficients in different bases.  In this paper we will answer a few conjectures of \cite{Sagan}. 

The coefficient sequence $a_k(I)$ is defined uniquely through the following equation
\[
d(I,n)=\sum_{k=0}^m a_k(I){n-m\choose k}.
\]
In \cite{Sagan} it was shown that the sequence $a_k(I)$ is non-negative, since it counts some combinatorial objects.  By taking a transformation of this sequence we were able to apply Stanley's theorem about the statistics of heights of a fixed element in a poset. As a result we prove
\begin{repTh}{thm:a_base}
 If $I\neq \emptyset$, then the sequence $\{a_k(I)\}_{k=0}^m$ is log-concave, that means that for any $0<k<m$ we have
 \[
 a_{k-1}(I)a_{k+1}(I)\le a_k^2(I).
 \]
\end{repTh}
As a corollary of the  proof of Theorem~\ref{thm:a_base} we get a bound on the roots of $d(I,n)$:
\begin{repTh}{cor:root_a}
 If $I\neq \emptyset$ and $d(I,z_0)=0$ for some $z_0\in \mathbb{C}$, then $|z_0|\le m$.
\end{repTh}

As in \cite{Sagan} we will also consider the $c_k(I)$ coefficient sequence, that is defined by the following equation
\[
d(I,n)=\sum_{k=0}^m(-1)^{m-k}c_k(I){n+1\choose k}.
\]
By using a new recursion from \cite{Sagan} we prove that
\begin{repPro}{prop:c_base}
 If $I\neq \emptyset$, then for any $0\le k\le m$ the coefficient $c_k(I)\ge 0$.
\end{repPro}

In the last section we will establish zero-free regions for descent polynomials. In particular we will prove the following.
\begin{repTh}{th:root_main}
 If $I\neq \emptyset$ and $d(I,z_0)=0$ for some $z_0\in \mathbb{C}$, then $|z_0-m|\le m+1$. In particular, $\Re z_0 \ge -1$.
\end{repTh}

This paper is organized as follows. In the next section we will define two sequences, $a_k(I)$ and $c_k(I)$, we recall the two main recursions for the descent polynomial and we introduce one of our main key ingredients. Then in Section~\ref{sec:c_base} we will prove a conjecture concerning the sequence $c_k(I)$ and some consequences. In Section~\ref{sec:a_base} we will prove a conjecture concerning the sequence $a_k(I)$, then in Section~\ref{sec:root} we prove some bounds on the roots.

\section{Preliminaries}

In this section we will recall some recursions of the descent polynomial and we will establish some related coefficient sequences by choosing different bases for the polynomials.

First of all, for the rest of the paper we will always denote  a finite subset of $\mathbb{Z}^+$ by $I$, and $m(I)$ is the maximal element of $I\cup\{0\}$. If it is clear from the context, $m(I)$ will be denoted by $m$.

Let us define the coefficients $a_k(I),c_k(I)$ for any $I$ with maximal element $m$ and $k\in\mathbb{N}$ through the following expressions:
\[
  d(I,n)=\sum_{k=0}^m a_k(I) {n-m\choose k} = \sum_{k=0}^m (-1)^{m-k}c_k(I){n+1\choose k},
\]
if $k\le m$, and $c_k(I)=a_k(I)=0$, if $k>m$. Observe that they are well-defined, since $\{{n-m\choose k}\}_{k\in\mathbb{N}}$ and also $\{{n+1\choose k}\}_{k\in\mathbb{N}}$ form a base of the space of one-variable polynomials. For later on, we will refer to the first and second bases as ``$a$-base'' and ``$c$-base'', respectively. We will also consider an other base that is also a Newton-base.

As it turns out, these coefficients are integers, moreover, they are non-negative. To be more precise, in \cite{Sagan} it has been proved that $a_k(I)$ counts some combinatorial objects (i.e. they are non-negative integers), and $c_0(I)$ is non-negative. The authors of \cite{Sagan} also conjectured that each $c_k(I)\ge 0$, and for a proof of the affirmative answer see Proposition~\ref{prop:c_base}. 

Next, we would like to establish two recurrences for the descent polynomial, which will be intensively used in several proofs. Before that, we need the following notations. For an $\emptyset\neq I=\{i_1,\dots,i_l\}$ and $1\le t\le l$, let
\begin{align*}
  I^-=&I-\{i_l\},\\
  I_t=&\{i_1,\dots, i_{t-1}, i_t-1, \dots, i_l-1\}-\{0\},\\
  \widehat{I}_t=&\{i_1,\dots, i_{t-1}, i_{t+1}-1, \dots, i_l-1\},\\
  I'=&\{i_j~|~i_j-1\notin I\},\\
  I''=&I'-\{1\}.
\end{align*}

For the rest, $m(I)$ denotes the maximal element of a non-empty set $I\cup\{0\}$. If it is clear from the context, we will denote this element by $m$ .

\begin{Pro}\label{cor:rec_a}
 If $I\neq \emptyset$, then
 \[
  d(I,n)={n\choose m}d(I^-,m)-d(I^-,n)
 \]
\end{Pro}
In contrast to the simplicity of this recursion, the disadvantage is that the descent polynomial of $I$ is a difference of two polynomials. In \cite{Sagan}, the authors found an other way to write $d(I,n)$ as a sum of polynomials (Thm 2.4. of \cite{Sagan}). Now we will state an equivalent form, which will fit our purposes better, and we also give its proof.

\begin{Cor}\label{cor:rec}
 If $I\neq \emptyset$, then
 \begin{align}\label{rec}
    d(I,n+1)&=\\&d(I,n)+\sum_{i_t\in I''\setminus\{m\}}d(I_t,n)+\sum_{i_t\in I'\setminus\{m\}}d(\hat{I}_t,n)+d(I^-,m-1){n\choose m-1}.\nonumber
 \end{align}
\end{Cor}
\begin{proof}
 Let us recall the formula of Theorem 2.4. of \cite{Sagan}:
 \begin{eqnarray}\label{s_rec}
  d(I,n+1)=d(I,n)+\sum_{i_t\in I''}d(I_t,n)+\sum_{i_t\in I'}d(\hat{I}_t,n).
 \end{eqnarray}
 
 If $I=\{1\}$, then trivially \eqref{rec} is true. For $I\neq\{1\}$ we will distinguish two cases.

 If $m\notin I'$ (and also $m\notin I''$), then by definition it means that $m-1\in I$. But it means that $m-1\in I^-$ and 
 \[
  d(I^-,m-1){n\choose m-1}=0.
 \]
 Therefore the right hand side of \eqref{rec} is the same as the right hand side of \eqref{s_rec}.
 
 If $m\in I'$ (and also $m\in I''$), then $i_l=m$, $\hat{I}_l=I^-\cup \{m-1\}$ and $I_l=I^-$. Now take the difference of the right hand sides of \eqref{rec} and \eqref{s_rec}, that is 
 \begin{gather*}
  d(I^-,m-1){n\choose m-1}-d(I_l,n)-d(\hat{I}_l,n)=\\
  d(I^-,m-1){n\choose m-1}-\left(d(\hat{I}_l^-,n){n\choose m-1}-d(I^-,n)\right)-d(I^-,n)=\\
  d(I^-,m-1){n\choose m-1}-d(I^-,n){n\choose m-1}=0.
 \end{gather*}
 Therefore the two equations have to be equal.
\end{proof}

As a conjecture in \cite{Sagan} it arose that the coefficient sequence $\{a_k(I)\}_{k=0}^m$ is log-concave. We mean by that that for any $0<k<m$ we have
\[
  a_{k-1}(I)a_{k+1}(I)\le a_k(I)^2.
\]
In particular, the sequence $\{a_k(I)\}_{k=0}^m$ is unimodal.

Our main tool to attack this problem will be a result of Stanley about the height of a certain element of a finite poset in all linear extensions. So let $P$ be a finite poset and $v\in P$ a fixed element, and denote the set of order-preserving bijection from $P$ to the chain $[1,2,\dots,|P|]$ by $\textrm{Ext}(P)$. Then, the height polynomial of $v$ in $P$ defined as
\[
  h_{P,v}(x)=\sum_{\phi\in \textrm{Ext}(P)}x^{\phi(v)-1} = \sum_{k=0}^{|P|-1}h_k(P,v)x^k.
\]
In other words $h_k(P,v)$ counts how many linear extensions $P$ has, such that below $v$ there are exactly $k$ many elements. 

In special cases, when all comparable elements from $v$ (except for $v$) are bigger in $P$, we can reformulate $h_k(P,v)$ as it counts how many linear extensions $P$ has, such that below $v$ there are exactly $k$ many incomparable elements. For such a case, we could combine two results of Stanley to obtain the following theorem.

\begin{Th}\label{thm:stanley}
 Let $P$ be a finite poset, and $v\in P$ be fixed. Then the coefficient sequence $\{h_k(P,v)\}_{k=0}^{|P|-1}$ is log-concave. Moreover if all comparable elements with $v$ are bigger than $v$ in $P$, then $\{h_k(P,v)\}_{k=0}^{|P|-1}$ is a decreasing, log-concave sequence.
\end{Th}
\begin{proof}
 The first part of the theorem is Corollary 3.3. of \cite{STANLEY198156}. For the second part we use fact that $h_k(P,v)$ can be interpreted as the number of linear extensions such that there are $k$ many smaller than $v$ incomparable elements in the extension. Then by Theorem 6.5. of \cite{Stanley1986} we obtain the desired statement.
\end{proof}

We will use this theorem in a special case. For any $I$ we define a poset $P_I$ on $[u_1,\dots, u_{m+1}]$, as $u_i >u_{i+1}$ if $i\in I$ and $u_i<u_{i+1}$ if $i\notin I$.
Observe that any comparable element with $x_{m+1}$ is bigger in $P_I$, therefore the sequence $\{h_k(P_I,u_{m+1})\}_{k=0}^{m}$ is  decreasing and log-concave. We would like to remark that any linear extension of $P_I$ can be viewed as an element of $D(I,m+1)$. In that way we can write that
\[
  h(I,x)=h_{P_I,u_{m+1}}(x)=\sum_{\pi\in D(I,m+1)}x^{\pi_{m+1}-1}.
\]

\section{Descent polynomial in ``$c$-base''}\label{sec:c_base}

The aim of the section is to give an affirmative answer for Conjecture 3.7. of \cite{Sagan}, and give some immediate consequences on the coefficients and evaluation. For corollaries considering the roots of $d(I,n)$ see Section~\ref{sec:root}. We would like to remark at that point that the proof will be just an algebraic manipulation, not a ``combinatorial'' proof. However, giving such a proof could imply some kind of ``combinatorial reciprocity'' for descent polynomials.

First, we will translate the recursion of Corollary \ref{cor:rec} to the terms of $c_k(I)$.

\begin{Lemma}\label{lem:c_rec}
 If $I\neq \emptyset$ and $0\le k\le m-1$, then
 \[
  c_{k+1}(I)=\sum_{i_t\in I''\setminus\{m\}}c_k(I_t)+\sum_{i_t\in I'\setminus\{m\}}c_k(\widehat{I}_t)+d(I^-,m-1).
 \]
\end{Lemma}
\begin{proof}

 The idea is that we rewrite the equation of \ref{cor:rec} as
 \[
  d(I,n+1)-d(I,n)=\sum_{i_t\in I''\setminus\{m\}}d(I_t,n)+\sum_{i_t\in I'\setminus\{m\}}d(\hat{I}_t,n)+d(I^-,m-1){n\choose m-1},
 \]
 and express both sides in $c$-base, then compare the coefficients of ${n+1\choose k}$.

 The left side can be written as
 \begin{gather*}
  d(I,n+1)-d(I,n)=\\ 
  \sum_{k=0}^mc_k(I)(-1)^{m-k}{n+2\choose k}-\sum_{k=0}^mc_k(I)(-1)^{m-k}{n+1\choose k}=\\
  \sum_{k=1}^mc_k(I)(-1)^{m-k}{n+1\choose k-1}=\\
  \sum_{k=0}^{m-1}c_{k+1}(I)(-1)^{m-k-1}{n+1\choose k}.
 \end{gather*}
 
 Next we use the famous Chu-Vandermonde's identity:  
 \[
  {n\choose m-1}=\sum_{k=0}^{m-1}{n+1\choose k}{-1\choose m-1-k}=\sum_{k=0}^{m-1}(-1)^{m-1-k}{n+1\choose k}.
 \]
 Therefore the right hand side can be written as:
 \begin{gather*}
  \sum_{i_t\in I''\setminus\{m\}}d(I_t,n)+\sum_{i_t\in I'\setminus\{m\}}d(\widehat{I}_t,n)+d(I^-,m-1){n\choose m-1}=\\
  \sum_{k=0}^{m-1}(-1)^{m-1-k}\left(\sum_{i_t\in I''\setminus\{m\}}c_{k}(I_t)+\sum_{i_t\in I'\setminus\{m\}}c_{k}(\widehat{I}_t)+d(I^-,m-1)\right){n+1\choose k}.
 \end{gather*}
 
 We gain that for any $0\le k\le m-1$,
 \begin{align*}
  (-1)^{m-k-1}&c_{k+1}(I)=\\
  &(-1)^{m-1-k}\left(\sum_{i_t\in I''\setminus\{m\}}c_{k}(I_t)+\sum_{i_t\in I'\setminus\{m\}}c_{k}(\widehat{I}_t)+(-1)^{m-1}d(I^-,m-1)\right).
 \end{align*}
 By multiplying both sides by $(-1)^{m-k-1}$ we get the desired statement.
\end{proof}

Similarly, we can rephrase Proposition~\ref{cor:rec_a}, but we leave the proof for the readers.
\begin{Lemma}
 If $I\neq\emptyset$ and $0\le k\le m$, then
 \begin{eqnarray}\label{eq:1}
  c_k(I)=d(I^-,m)-(-1)^{m-m^-}c_k(I^-),
 \end{eqnarray}
 where $m^-=m(I^-)$.
\end{Lemma}

The next theorem  settles Conjecture~3.7 of \cite{Sagan}. We would like to point out that the non-negativity of  $c_0(I)$ has already been proven in \cite{Sagan}, and one can use it to find a shortcut in the proof. However, we will give a self-contained proof.

\begin{Th}\label{prop:c_base}
 For any $I$ and $0\le k\le m$, the coefficient $c_k(I)$ is a non-negative integer.
\end{Th}
\begin{proof}
 We will proceed by induction on $m$.
 If $m=0$, then $I=\emptyset$, thus,
 \[
  d(I,n)=1,
 \]
 therefore $c_0(I)=1\ge 0$.

  If $m=1$, then $I=\{m\}$ and 
  \begin{gather*}
   d(I,n)={n\choose m}-1=\sum_{k=0}^m {n+1 \choose k}{-1\choose m-k}- {n+1\choose 0}=\\
   \sum_{k=1}^m (-1)^{m-k}{n+1\choose k} + (-1)^{m-0}{n+1 \choose 0} (1-(-1)^m).
  \end{gather*}
  We obtained that
  \[
   c_k(I)=\left\{\begin{array}{lc}
                 1&\textrm{if $0<k\le m$}\\
                 2 & \textrm{if $k=0$ and $m$ is odd}\\
                 0 & \textrm{if $k=0$ and $m$ is even}
                \end{array}
   \right.
  \]
  
  For the rest of the proof, we assume that the size of $I$ is at least 2. Therefore $m>1$, and $m^-=\max(I^-)>0$.
  Since for any $i_t\in I''$ (and $i_t\in I'$) the maximum of $I_t$ (and $\widehat{I}_t$) is exactly $m-1$, we can use induction on them, i.e. $c_k(I_t)\ge 0$ integer ($c_k(\widehat{I}_t)\ge 0$ integer). On the other hand, $d(I^-,m-1)$ counts permutations with descent set $I^-$, so $d(I^-,m-1)\ge 0$ integer. 
  Now by Lemma \ref{lem:c_rec} and by the previous paragraph we have for any $k\ge 1$ that 
  \begin{eqnarray}
    c_k(I)=\sum_{i_t\in I'\setminus\{m\}}c_{k-1}(I_t)+\sum_{i_t\in I''\setminus\{m\}}c_{k-1}(\widehat{I}_t)+d(I^-,m-1)\ge 0. \label{eq:2}
  \end{eqnarray}
 
  What remains is to prove that $c_0(I)\ge 0$. This is exactly the statement of Proposition 3.10. of \cite{Sagan}, but for the completeness we also give its proof.
  
  We consider two cases. If $m-1\in I$, then by \eqref{eq:1} 
  \[
    c_0(I)=d(I^-,m)-(-1)^{m-(m-1)}c_0(I^-)=d(I^-,m)+c_0(I^-)\ge 1+0>0,
  \]
  since $m>\max(I^-)$.
  
  If $m-1\notin I$, then by \eqref{eq:2},
  \[
    c_1(I)\ge d(I^-,m-1)\ge 1.
  \]
  
  On the other hand, we can express $d(I,0)$ in two ways. The first equality is by Lemma 3.8. of \cite{Sagan}, the second is by the definition of $c_k(I)$.
  \[
    (-1)^{\# I}=d(I,0)=\sum_{k=0}^m(-1)^{m-k}c_k(I){1\choose k}=(-1)^m(c_0(I)-c_1(I)),
  \]
  therefore
  \[
    c_0(I)=c_1(I)+(-1)^{\# I+m}\ge 1+(-1)=0. 
  \]
 
\end{proof}

As a corollary we will see that the values of the polynomial $d(I,n)$ at negative integers are of the same sign. This phenomenon is kind of similar to a ``combinatorial reciprocity'', by which we mean that there exists a sequence of ``nice sets'' $A_n$ parametrized by $n$, such that $(-1)^md(I,-n)=|A_n|$. We think that either proving the previous theorem using combinatorial arguments or finding a combinatorial reciprocity for $d(I,n)$  could provide an answer for the other.

\begin{Cor}
 Let $n$ be a positive integer, then
 \[
  (-1)^md(I,-n)\ge 0.
 \]
 Moreover if $n>1$ positive integer, then $(-1)^md(I,-n)>0$.
\end{Cor}
\begin{proof}
 Assume that $n=1$. Then 
 \[
  (-1)^md(I,-1)=\sum_{k=0}^m(-1)^{-k}c_k(I){-1+1\choose k}=(-1)^0c_0(I){0\choose 0}=c_0(I),
 \]
 and by the previous proposition we know that $c_0(I)\ge 0$.
 \begin{gather*}
  (-1)^m d(I,-n)=\sum_{k=0}^m c_k(I)(-1)^{-k}{-n+1 \choose k}=\\
  \sum_{k=0}^m c_k(I)(-1)^{-k}(-1)^k{n+k-2 \choose k}=\sum_{k=0}^mc_k(I){n+k-2\choose k}> 0
 \end{gather*}
\end{proof}

We would like to remark that in Section~\ref{sec:root} we will prove that in particular there is no root of $d(I,n)$ on the half-line $(-\infty, -1)$, that is, for any real number $z_0\in(-\infty,-1)$, the expression $(-1)^md(I,z_0)$ is always positive.

Moreover if we carefully follow the previous proofs, then one might observe that $d(I,-1)=0$ iff $c_0(I)=0$ iff  $I=\{m\}$ where $m$ is even or $I=[m-2]\cup\{m\}$.

\section{Descent polynomial in ``$a$-base''}\label{sec:a_base}

In this section we would like to investigate the coefficients $a_k(I)$. In order to do that, we will need to understand the coefficients of $d(I,n)$ in the base of $\{{n-m+k\choose k+1}\}_{k=-1}^{m-1}$, which is defined by the following equation
\[
  d(I,n)=\overline{a}_{-1}(I){n-m-1\choose 0}+\overline{a}_0(I){n-m\choose 1}+\dots+\overline{a}_{m-1}(I){n-1\choose m}.
\]
Observe that $\overline{a}_{-1}(I)=0$, since 
\[
  0=d(I,m)=\overline{a}_{-1}(I){-1\choose 0}+\sum_{k=0}^{m-1}\overline{a}_k(I){k\choose k+1}=\overline{a}_{-1}(I),
\]
therefore later on, we will concentrate on the coefficients $\overline{a}_k(I)$ for $0\le k\le m-1$. As it will turn out, all these coefficients are non-negative integers, moreover, each of them counts some combinatorial objects.

On the other hand, this new coefficient sequence is closely  related to the coefficients $a_k(I)$. To show the connection, we introduce two polynomials
\begin{gather*}
 a(I,x)=\sum_{k=0}^ma_k(I)x^k,\\
 \overline{a}(I,x)=\sum_{k=0}^{m-1}\overline{a}_k(I)x^k.
\end{gather*}

First we will show that $\overline{a}_k(I)=h_{m-k}(P_I,u_{m+1})$, i.e. $\overline{a}_k(I)$ counts the number of permutations from $D(I,m+1)$, such that there are $(k+1)$ elements above $u_{m+1}$.

\begin{Pro}
If $I\neq \emptyset$ and $0\le k\le m-1$, then
\[
  \overline{a}_k(I)=h_{m-k}(P_I,u_{m+1}).
\]
\end{Pro}
\begin{proof}
 We will show that if $n>m$, then \[d(I,n)=\sum_{k=0}^{m-1}h_{m-k}(P_I,u_{m+1}){n-m+k\choose k+1}.\] It is enough, since $\{{n-m+k\choose k+1}\}_{k=-1}^{m-1}$ is a base in the space of polynomials of degree at most $m$.
 
 Let us define the sets $B_k(I,n)=\{\pi\in D(I,n)~|~\pi_{m+1}=k\}$ for $1\le k\le n$. 
 For any $\pi\in D(I,n)$ the last descent is between $m$ and $m+1$, therefore $\pi_{m}>\pi_{m+1}<\pi_{m+2}<\dots<\pi_{n}\le n$, i.e. $\pi_{m+1}\le m$.
 Therefore $B_k(I,n)=\emptyset$ for any $m<k\le n$, and $D(I,n)$ is a disjoint union of the sets $B_k(I,n)$ for $1\le k\le m$. Also observe that $|B_k(I,m+1)|=h_k(P_I,u_{m+1})$.
 
 We claim that \[|B_k(I,n)|=|B_k(I,m+1)\times {[k+1,n]\choose m+1-k}|=|B_k(I,m+1)|{n-k\choose m+1-k}.\]
 To prove the first equality we establish a bijection. If $\pi\in B_k(I,n)$, then let $E_\pi=\{1\le i\le m~|~\pi_i>k\}$, $V_\pi=\{\pi_i~|~i\in E_\pi\}$ and $\pi|_{m+1}\in B_k(I,m+1)$ the unique induced linear ordering on the first $m+1$ element. As before, for any $l>m+1$ the value $\pi_l$ is bigger than $\pi_{m+1}$, therefore $|E_\pi|=m+1-k$ and $V_\pi\subseteq [k+1,n]$ has size $m+1-k$.
 So let $f:B_k(I,n)\to B_k(I,m+1)\times {[k+1,n]\choose m+1-k}$ defined as
 \[
  f(\pi)=\left(\pi|_{m+1},V_{\pi}\right).
 \]
 Checking whether the function $f$ is a bijection is left to the readers.
 
 Putting the pieces together, we have 
 \begin{gather*}
  d(I,n)=|D(I,n)|=|\cup_{k=1}^m B_k(I,n)|=\sum_{k=1}^m |B_k(I,n)|=\\
  \sum_{k=1}^m |B_k(I,m+1)\times {[k+1,n]\choose m+1-k}|=\\
  \sum_{k=1}^m |B_k(I,m+1)|{n-k\choose m+1-k}=\\
  \sum_{k=1}^m h_k(P_I,u_{m+1}){n-k\choose m+1-k}=\\
  \sum_{l=0}^{m-1} h_{m-l}(P_I,u_{m+1}){n-m+l\choose l+1}.
 \end{gather*}
\end{proof}

\begin{Cor}\label{cor:bara_logconcave}
 If $ I\neq \emptyset$, then  the sequence $\overline{a}_0(I),\overline{a}_1(I),\dots, \overline{a}_{m-1}(I)$ is a monotone increasing, log-concave sequence of non-negative integers.
\end{Cor}
\begin{proof}
 By the previous proposition we know that this sequence is the same as $\{h_{m-k}(P_I,u_{m+1})\}_{k=1}^m$, which is clearly a sequence of non-negative integers. Moreover, by Theorem \ref{thm:stanley}, it is log-concave and monotone decreasing.
\end{proof}

We just want to remark that since the polynomial $\overline{a}(I,x)$ has a monotone coefficient sequence, all of its roots are contained in the unit disk (see Figure~\ref{fig:bara}). 

\begin{figure}[h]
   \centering
   \includegraphics[width=6.5cm]{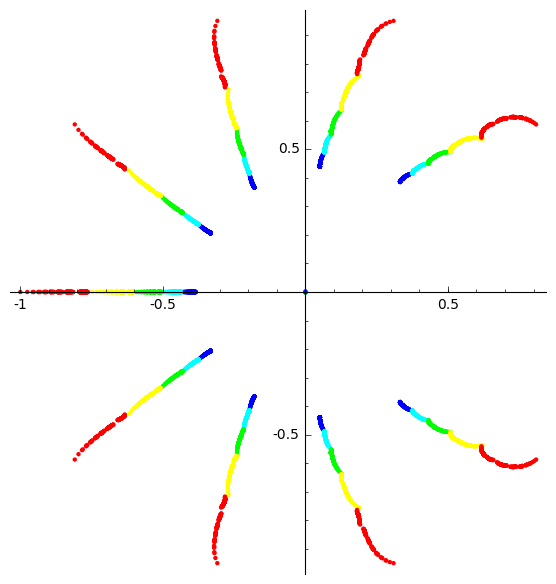}
   \caption{The roots of $\bar{a}(I,n)$ where $I$ has the form $I=J\cup [10,11,\dots,10+k]$ for some $k=0,\dots,4$ and $J\subseteq [8]$. Different colors mark different values of $k$.}
   \label{fig:bara}
\end{figure}

Our next goal is to establish a connection between the coefficients $a_k(I)$ and $\overline{a}_k(I)$.

\begin{Pro}
 If $I\neq \emptyset$, then 
 \[
  a(I,x)=x\overline{a}(I,x+1)
 \]
\end{Pro}
\begin{proof}
  By definition we see that 
 \begin{gather*}
  d(I,n)=\sum_{k=0}^{m-1} \overline{a}_k(I){n-m+k\choose k+1}=\sum_{k=0}^{m-1} \overline{a}_k(I)\sum_{l=0}^{k+1}{n-m\choose l}{k\choose k+1-l}=\\
  \sum_{k=0}^{m-1} \overline{a}_k(I)\sum_{l=1}^{k+1}{n-m\choose l}{k\choose l-1}=
  \sum_{l=1}^{m}{n-m\choose l}\left(\sum_{k=l-1}^{m-1}\overline{a}_k(I){k\choose l-1}\right),
 \end{gather*}
 which means that $a_l(I)=\sum_{k=l-1}^{m-1}\overline{a}_l(I){k\choose l-1}$ for $1\le l\le m$, i.e.
 \[
  a(I,x)=\sum_{l=1}^{m}x^l\left(\sum_{k=l-1}^{m-1}\overline{a}_k(I){k\choose l-1}\right)
 \]
 
 On the other hand, let us calculate the coefficients of $x\overline{a}(I,x+1)$. 
 \begin{gather*}
  x\overline{a}(I,x+1)=x\left(\sum_{k=0}^{m-1}\overline{a}_k(I)(x+1)^k\right)=\\
  x\left(\sum_{k=0}^{m-1}\overline{a}_k(I)\sum_{l=0}^{k}{k\choose l}x^l\right)=
  x\left(\sum_{l=0}^{m-1}x^l\sum_{k=l}^{m-1}\overline{a}_k(I){k\choose l}\right)=\\
  \sum_{l=0}^{m-1}x^{l+1}\sum_{k=l}^{m-1}\overline{a}_k(I){k\choose l}=\\
  \sum_{l=1}^{m}x^{l}\sum_{k=l-1}^{m-1}\overline{a}_k(I){k\choose l-1}=a(I,x).
 \end{gather*}
\end{proof}

As a corollary of two previous propositions, we will give a proof of Conjecture~3.4 of \cite{Sagan}.

\begin{Cor}\label{thm:a_base}
 If $I\neq\emptyset$, then the sequence $a_0(I),a_1(I), \dots, a_m(I)$ is a log-concave sequence of non-negative integers.
\end{Cor}
\begin{proof}
 By Corollary~\ref{cor:bara_logconcave} we know that the coefficient sequence of the polynomial $\overline{a}(I,x)$ is log-concave, and by monotonicity, it is clearly without internal zeros. Therefore by the fundamental theorem of \cite{Brenti1989unimodal}, the coefficient sequence of the polynomial $\overline{a}(I,x+1)$ is log-concave. Since multiplication with an $x$ only shifts the coefficient sequence, $x\overline{a}(I,x+1)=a(I,x)$ also has a log-concave coefficient sequence.
\end{proof}

\section{On the roots of $d(I,n)$}\label{sec:root}

In this section we will prove four propositions about the locations of the roots of $d(I,n)$, two are for general $I$, and two are for some special ones. The first result is obtained by the technique of  Theorem 4.16. of \cite{Sagan} based on the non-negativity of the coefficients $c_k(I)$. In the second, we will prove a linear bound in $m$ for the length of the roots of $d(I,n)$, which will be based on the monotonicity of the coefficients $\overline{a}_k(I)$. For the third we use similar arguments as in the proof of the second statement. In the fourth we will prove a real-rootedness for some special $I$ using Neumaier's Gershgorin type result.

First we will recall some basic notations from \cite{Sagan}. Let $R_m$ be the region described by Theorem 4.16. of \cite{Sagan}, that is $R_m=S_m\cup \overline{S_m}$ and 
\[
  S_m=\{z\in\mathbb{C}~|~\arg(z)\ge 0 \textrm{ and } \sum_{i=1}^m\arg(z-i+1)<\pi\}.
\]
Then we have the following corollary of Proposition~\ref{prop:c_base}.

\begin{Cor}
 Let $I$ be a finite set of positive integers. Than any element of $(m-2)-R_m$ is not a root of $d(I,z)$. In particular, if $z_0$ is a real root of $d(I,z)$, then $z_0\ge -1$.
\end{Cor}
\begin{proof}
 Let $z\in\mathbb{C}$ be a complex number such that  
 \[S=\{(-1)^0(z+1)_{\downarrow 0},\dots, (-1)^m(z+1)_{\downarrow m}\}\]
 is non-negatively independent, i.e.
 \[S=\{(-1-z)_{\uparrow 0},\dots, (-1-z)_{\uparrow m}\}\]
 is in an open half plane $H$, such that $1\in H$. But this is equivalent to the fact that the points
 \[
   S'=\{(m-2-z)_{\downarrow m}(-1-z)_{\uparrow 0}^{-1}, \dots, (m-2-z)_{\downarrow m}(-1-z)_{\uparrow m}^{-1}\}
 \]
 are in $H$, which is the same set as
 \[
  S'=\{(m-2-z)_{\downarrow m},(m-2-z)_{\downarrow m-1} \dots, (m-2-z)_{\downarrow 0}\}.
 \]
 But by Theorem 4.16. of \cite{Sagan}, we know that this set lies on an open half-plane iff $m-2-z\in R_m$.
 
 Therefore $S$ is an open half plane iff $m-2-z\in R_m$ iff $z\in(m-2)-R_m$.
 
 The last statement can be obtained from the fact that $(m-1,\infty)\subseteq R_m$.
\end{proof}

The following lemma will be useful in the upcoming proofs.

\begin{Lemma}\label{lem:convex}
 Let $m>0$ integer given and assume that $|z|>m$. Then the lengths
 \[
  \left|{z-m+k\choose k}\right|
 \]
 are increasing for $k=0,\dots,m$. 
 
 In particular,  if $\alpha_0,\dots,\alpha_{m}\in \mathbb{R}$, $\alpha_m\neq 0$,  $\sum_{i=0}^{m-1}|\alpha_i|\le |\alpha_m|$  and $|z|>m$, then 
 \[
  \left|\alpha_m{z\choose m}\right| > \left|\sum_{k=0}^{m-1}\alpha_k{z-m+k\choose k}\right|.
 \] 
 \[
  \sum_{i=0}^m\alpha_k{z-m+k\choose k}\neq 0
 \]
\end{Lemma}
\begin{proof}
 Let $0\le k\le m-1$ be fixed. Then to see that the lengths are increasing we have to consider the ratio of two consecutive elements:
 \begin{align*}
  \left|{z-m+k+1\choose k+1}\right|& \left|{z-m+k\choose k}\right|^{-1}=\\
    &=\frac{|z-m+k+1|}{k+1}\ge \frac{|z|-m+k+1}{k+1}>1
 \end{align*}
 Therefore the sequence is increasing.
 
 To see the second statement let us define $C=\sum_{i=0}^{m-1}|\alpha_i|$. If $C=0$, then the statement is trivially true. If $C\neq 0$, then the vector 
 \[
  v=\left|\sum_{k=0}^{m-1}\frac{\alpha_k}{C}{z-m+k\choose k}\right|=\left|\sum_{k=0}^{m-1}\frac{|\alpha_k|}{C}\textrm{sign}(\alpha_i){z-m+k\choose k}\right|
 \]
 is a convex combination of the vectors $\left\{\textrm{sign}(\alpha_k){z-m+k\choose k}\right\}_{k=0}^{m-1}$. Hence
 \[
  |v|\le \left|{z-1\choose m-1}\right|,
 \]
 and
 \[
  \left|\sum_{k=0}^{m-1}\alpha_k{z-m+k\choose k}\right|=C|v|\le C\left|{z-1\choose m-1}\right|<\alpha_m\left|{z\choose m}\right|
 \]
\end{proof}

\begin{Cor}\label{cor:root_a}
 If $z_0$ is a root of $d(I,z)$, then $|z_0|\le m$.
\end{Cor}

\begin{proof}
 Let us consider the polynomial $p(z)=(z-1)\bar{a}(I,z)$, and let $p_i$ (resp. $\bar{a}_i$) be the coefficient of $z^i$ in $p$ (resp. $\bar{a}$), i.e.
 \[
  p(z)=\sum_{i=0}^mp_iz^i\qquad \bar{a}(I,z)=\sum_{i=0}^{m-1}\bar{a}_i z^i.
 \]
 The relation of $p$ and $\bar{a}$ translates as follows:
 \[
  p_i=\left\{\begin{array}{lc}
              \bar{a}_{m-1} & \textrm{if $i=m$}\\
              \bar{a}_{i-1}-\bar{a}_{i} & \textrm{if $0<i<m$}\\
              -\bar{a}_0 & \textrm{if $i=0$}
             \end{array}
\right.
 \]
 and
 \[
  d(I,n)=\sum_{k=0}^mp_k{n-m+k\choose k}.
 \]
 Since the coefficient sequence of $\bar{a}(I,z)$ is non-decreasing by Corollary~\ref{cor:bara_logconcave}, therefore all coefficients of $p$ except $p_m$ are non-positive and their sum is 0. In other words for any $k\in\{0,1,\dots,m-1\}$:
 \[
  |p_k|=-p_k
 \]
 and 
 \[
  \sum_{k=0}^{m-1}|p_k|=-\sum_{k=0}^{m-1}p_k=a_{m-1}=p_m>0.
 \]
 Therefore by Lemma~\ref{lem:convex}, if $|z|>0$, then 
 \[
  d(I,z)=\sum_{k=0}^mp_k{z-m+k\choose k}\neq 0.
 \]
 
%
%
\end{proof}


In the previous proof we did not use the fact that $\bar{a}_k(I)$ is a log-concave sequence, which would be interesting if one could make use of it. 
Our next goal is to prove Theorem~\ref{th:root_main}. In order to prove it, we have to distinguish a few cases depending on the number of consecutive elements ending at $\max(I)$. For simplicity, first we will consider the case, when the distance of the last two elements is at least 2.

\begin{Pro}\label{pro:c}
 If $I=\{i_1,\dots,i_l\}$ for some $l\ge 1$, such that $|I|=1$ or $i_l-i_{l-1}\ge 2$. If $d(I,z_0)=0$, then
 \[
  |m-1-z_0|\le m.
 \]
 In particular $\Re z_0\ge -1$.
\end{Pro}
\begin{proof}
 Let us consider $p(n)=(-1)^md(I,m-1-n)$ using coefficients $c_k(I)$.
 \begin{gather*}
  d(I,-(n-m+1))=\sum_{k=0}^m (-1)^{m-k}c_k(I){-(n-m+1)+1\choose k}=\\
  \sum_{k=0}^m (-1)^{m-k}c_k(I)(-1)^k{n-m+k-1\choose k}=\\
  (-1)^m\sum_{k=0}^m c_k(I){n-m+k-1\choose k}.
 \end{gather*}
 It might be familiar from the proof of Corollary~\ref{cor:root_a}. As before we expend $p(n)$ in base $\{{n-m+k\choose k}\}_{k\in \mathbb{N}}$.
 \begin{gather*}
  p(n)=\sum_{k=0}^m c_k(I){n-m+k-1\choose k}=\\
  \sum_{k=1}^m c_k(I)\left({n-m+k\choose k}-{n-m+k-1\choose k-1}\right) + c_0(I){n-m-1\choose 0}=\\
  c_m{n\choose m}+\sum_{k=0}^{m-1} (c_k(I)-c_{k+1}(I)){n-m+k\choose k}=\\
  \sum_{k=0}^m\tilde{c}_k(I){n-m+k\choose k}.
 \end{gather*}

 Now we claim that $\sum_{k=0}^{m-1}|\tilde{c}_k(I)|\le c_m(I)$. To prove that, we use induction on $|I|$ and $m$, and we use the recursion of Lemma~\ref{lem:c_rec}. If $I=\{m\}$, then it can be easily checked. 
 
 So for the rest assume, that the statement is true for sets of size at most $l-1$ and with maximal element at most $m-1$. 
 Let $|I|=l\ge 2$ with $i_l=m$ and assume that $i_l-i_{l-1}\ge 2$.
 Then
 \begin{gather*}
  \sum_{k=0}^{m-1}|c_k(I)-c_{k+1}(I)|=\\
  |c_0(I)-c_1(I)|+\sum_{k=1}^{m-1}\left|\sum_{t\in I''\setminus\{m\}}c_{k-1}(I_t)-c_{k}(I_t)+\sum_{t\in I'\setminus\{m\}}c_{k-1}(\hat{I}_t)-c_{k}(\hat{I}_t)\right|\le\\
1+  \sum_{k=0}^{m-2}\sum_{t\in I''\setminus\{m\}}|c_{k}(I_t)-c_{k+1}(I_t)|+\sum_{k=0}^{m-2}\sum_{t\in I'\setminus\{m\}}|c_k(\hat{I}_t)-c_{k+1}(\hat{I}_t)|
 \end{gather*}
  For any $t\in I''\setminus\{m\}$ the two largest elements of $I_t$ will be $i_{t-1}-1$ and $i_{t}-1=m-1$, so their difference is at least 2, therefore we can use inductive hypothesis.
  If $t\in I'\setminus\{m\}$, then either $\hat{I}_t$ has exactly one element, or $|\hat{I}_t|>1$. In this second case the largest element of $\hat{I}_t$ is $i_t-1=m-1$ and the second largest is $i_{t-2}$ or $i_{t-1}-1$. Clearly in each cases the inductive hypothesis is true, therefore
  \begin{gather*}
    \sum_{k=0}^{m-1}|c_k(I)-c_{k+1}(I)|\le 1+\sum_{t\in I''\setminus\{m\}}c_{m-1}(I_t)+\sum_{t\in I'\setminus\{m\}}c_{m-1}(\hat{I}_t)=\\
    1+c_{m}(I)-d(I^-,m-1)\le c_{m}(I)=\tilde{c}_m(I).
  \end{gather*}
  The last inequality is true, since $m-1> \max(I^-)$.
  
  So we obtained that $\sum_{k=0}^{m-1}|\tilde{c}_k(I)|\le c_m(I)$, therefore by Lemma~\ref{lem:convex}, if $|z|>m$, then 
  \[
    0\neq \sum_{k=0}^m\tilde{c}_k(I){z-m+k\choose k}=p(z)=(-1)^m d(I,m-1-z),
  \]
  equivalently if $|m-1-z_0|>m$, then $d(I,z_0)\neq 0$.
  
%
%

\end{proof}

We would like to remark two facts about the previous proof. First of all the introduced ``new'' coefficients, $\tilde{c}_k(I)$,  are exactly 
\[
  \tilde{c}_k(I)=d(I^c,k) =\left\{ \begin{array}{cc}
                       (-1)^{m+|[k+1,\infty)\cap I|+k}d(I\cap[k-1],k)& \textrm{ if $k\in I$}\\
                       0 & \textrm{ otherwise}
                      \end{array}
\right.,
\]
where $I^c=[m]\setminus I$, therefore
\[
  d(I,n)=\sum_{k=0}^m (-1)^{m-k}\tilde{c}_k(I){n\choose k}=\sum_{k=0}^m(-1)^{m-k}d(I^c,k){n\choose k}.
\]

Secondly we can not extend the proof for any $I$, because the crucial statement, that was $\sum_{k=0}^{m-1}|c_k(I)-c_{k+1}(I)|\le c_m(I)$, is not true for any $I\subseteq \mathbb{Z}^+$. (E.g. $I=\{1,2,3,4,5\}$)

From now on we would like to understand the roots of $I$'s with ``non-trivial endings''. To analyses these cases we introduce for the rest of the paper the following notation: for any finite set $I\subseteq\mathbb{Z}^+$ and $t\in\mathbb{N}$ let $I^t=I\cup\{m+1,m+2,\dots, m+t\}$.

\begin{Pro}\label{pro:c_root_small}
 For any $\emptyset \neq I$ such that $m-1\notin I$. Then if $t=1,2,3,4$, then there exists an $m_0=m_0(t)$, such that if $m\ge m_0$ and $d(I^t,z_0)=0$, then
 \[
  |m+t-1-z_0|\le m+t.
 \]
\end{Pro}
\begin{proof}
 Let us consider $d(I^t,n)$ in base $\{{n\choose k}\}_{k\in \mathbb{N}}$. Then
 \[
  d(I^t,n)=\sum_{k=0}^{m+t}(-1)^{m+t-k}d(I^c,k){n\choose k},
 \]
 where $I^c=(I^t)^c=[m+t]\setminus I^t=[m]\setminus I$.
 
 We claim that if $t\in\{1,2,3,4\}$ and $m$ sufficiently large, then  for any  $m\le k<m+t$ we have
 \begin{eqnarray}\label{eq:cond}
  2d(I^c,k)\le d(I^c,k+1).
 \end{eqnarray}
 To see that let us observe that all the roots $\xi_1,\dots,\xi_{m-1}$ of $d(I^c,n)$ are in a ball of radious $m-1$ around 0 by Corollary~\ref{cor:root_a}. Without loss of generality let us assume that $\xi_1=\max(I^c)=m-1$. Then
 \begin{align*}
  \frac{d(I^c,k)}{d(I^c,k+1)}&=\left|\frac{d(I^c,k)}{d(I^c,k+1)}\right|=\frac{(k-\xi_1)\prod_{i=2}^{m-1}|k-\xi_i|}{(k+1-\xi_1)\prod_{i=2}^{m-1}|k+1-\xi_i|}\\
  &\frac{k-m+1}{k-m+2}\prod_{i=2}^{m-1}\frac{|k-\xi_i|}{|k+1-\xi_i|}\le \frac{k-m+1}{k-m+2}\prod_{i=2}^{m-1}\frac{k+m-1}{k+m}\\
  &\le \frac{t}{t+1}\left(\frac{2m+t-2}{2m+t-1}\right)^{m-2} = \frac{t}{t+1}\left(1-\frac{1}{2m+t-1}\right)^{m-2}\to\frac{t}{t+1}e^{-0.5}
 \end{align*}
 Since $\frac{t}{t+1}e^{-0.5}<1/2$, therefore we get that for any $t\in\{1,2,3,4\}$ there exists an $m_0=m_0(t)$, such that $\forall m\ge m_0$  and for any  $m\le k<m+t$ we have $2d(I^c,k)\le d(I^c,k+1)$. In particular $2^{m+t-k}d(I^c,k)\le d(I^c,m+t)$.
 
 To finish the proof let us assume that $m\ge m_0$ for some fixed $t\in\{1,2,3,4\}$.
 Then consider the following polynomial $p(n)=(-1)^{m+t}d(I,m+t-1-n)$ as in the previos proof
 \begin{align*}
  p(n)&=(-1)^{m+t}\sum_{k=0}^{m+t}(-1)^{m+t-k}d(I^c,k){m+t-1-n\choose k}\\
    &=\sum_{k=0}^{m+t}d(I^c,k){n-(m+t)+k\choose k}
 \end{align*}
 
 Assume that $z_0$ is a zero of $p(n)$ with length at least $m+t$ i.e.
 \[
  d(I^c,m+t){z_0\choose m+t}=\sum_{k=0}^{m+t-1}(-d(I^c,k)){z_0-(m+t)+k\choose k}
 \]
 By the previous proof we get that $\sum_{k=0}^{m-1}|d(I^c,k)|\le d(I^c,m)$, therefore
 \begin{align*}
  C=\sum_{k=0}^{m+t-1}|-d(I^c,k)|&\le d(I^c,m)+\sum_{k=m}^{m+t-1}d(I^c,k)\\
  &\le 2^{-t}d(I^c,m+t)+\sum_{k=m}^{m+t-1}2^{-(m+t-k)}d(I^c,m+t)\\
  &=d(I^c,m+t).
 \end{align*}
 
 But it means that $\frac{d(I^c,m+t)}{C}{z_0\choose m+t}$ is a convex combination of  $\mathcal{F}=\{\epsilon_k{z_0-(m+t)+k\choose k}\}_{k=0}^{m+t-1}$, where $\epsilon_k=\textrm{sgn}(-d(I^c,k))$. However this is a contradiction, since $\frac{d(I^c,m+t)}{C}\ge 1$ and ${z_0\choose m+t}$ is strictly longer than any member of the set $\mathcal{F}$.
 
\end{proof}

Trivial upper bounds on $m_0$ is the smallest $m'_0$, such that for any $m\in[m_0',\infty)$ we have
\begin{eqnarray}\label{eq:cond_small}
  \frac{t}{t+1}\left(1-\frac{1}{2m+t-1}\right)^{m-2}<1/2.
\end{eqnarray}
These values can be found in the following Table~\ref{table:values}.

\begin{Lemma}\label{lem:c}
 For any $\emptyset \neq I$, such that $m-1\notin I$ and 
 \begin{eqnarray}\label{eq:cond_big}
  (m-1)(2m+1)\le {t+m-1\choose t},
 \end{eqnarray}
 then 
 \[
 d(I^c,m)(2m+1)\le d(I^c,m+t)
 \]
 
\end{Lemma}
\begin{proof}
 First of all 
 \[
    d(I^c,m)=d(I,m)\le d(I^-,m-1)(m-1)=d((I^c)^-,m-1)(m-1),
 \]
 because any $\pi \in D(I,m)$ can be written uniquely as an element in $D(I^-,m-1)\times [1,m-1]$.
 
 On the other hand 
 \[
  d(I^c,m+t)\ge {t+m-1\choose t}d((I^c)^-,m-1),
 \]
 because the left hand side counts the number of elements in $D(I^c,m+t)$, while the right hand side is the number of elements $\pi$ in $D(I^c,m+t)$, such that $\pi_{m} =1$.
 
 Combining these inequalities and using the hypothesis we get the desired statement.
\end{proof}

\begin{Pro}\label{pro:c_root_big}
 For any $\emptyset \neq I$ such that $m-1\notin I$. If 
 \[
  (m-1)(2m+1)\le {t+m-1\choose t}
 \]
 and  $d(I^t,z_0)=0$, then
 \[
  |m+t-z_0|\le m+t+1.
 \]
\end{Pro}

\begin{proof}
 Let us consider the polynomial $p(n)=(-1)^{m+t}d(I^t,m+t-n)$
 \begin{align*}
  p(n)=&(-1)^{m+t}d(I^t,m+t-n)=\sum_{k=0}^{m+t}d(I^c,k){-t-m+n+k-1\choose k}=\\
  &=\sum_{k=0}^{m-1}d(I^c,k){n-m-t+k-1\choose k}+\sum_{k=m}^{m+t}d(I^c,k){n-m-t+k-1\choose k}\\
  &=u(n)+\sum_{k=m}^{m+t}d(I^c,k){n-m-t+k-1\choose k}.
\end{align*}
As a result of the proof of Proposition~\ref{pro:c} we get that if $|z|>m$, then
\[
  |u(z+t+1)|=\left|\sum_{k=0}^{m-1}d(I^c,k){z-m+k\choose k}\right|< \left|d(I^c,m){z\choose m}\right|.
\]
So if $|z|>m+t+1$, then $|z-(t+1)|>m$ and therefore
\begin{align*}
  |u(z)|&\le d(I^c,m)\left|{z-t-1\choose m}\right|\\
  &\le d(I^c,m) \left(\left|{z-t\choose m}\right|+ \left|{z-t-1\choose m-1}\right|\right)\\
  &= d(I^c,m) \left(\left|\frac{(m+t)\dots(m+1)}{z(z-1)\dots(z-t+1)}\right|+ \left|\frac{(m+t)\dots m}{z(z-1)\dots(z-t)}\right| \right) \left|{z\choose m+t}\right|\\
  &<d(I^c,m) \left(\frac{2m+1}{t+m+1}\right)\left|{z\choose m+t}\right|
\end{align*}
Let us assume that $p(z)=0$ and $|z|>m+t+1$, therefore 
\[
  d(I^c,m+t){z\choose m+t} = \sum_{k=m-1}^{m+t-1}\left(d(I^c,k+1)-d(I^c,k)\right){z-m-t+k\choose k}+u(z),
\]
equivalently
\[
  {z\choose m+t} = \sum_{k=m-1}^{m+t-1}\frac{d(I^c,k+1)-d(I^c,k)}{d(I^c,m+t)}{z-m-t+k\choose k}+\frac{1}{d(I^c,m+t)}u(z). 
\]
Observe that the summation on the right hand side is a convex combination of some complex numbers, therefore its length can be bounded from above by the length of the longest vector, that is 
\begin{align*}
  \left|\sum_{k=m-1}^{m+t-1}\frac{d(I^c,k+1)-d(I^c,k)}{d(I^c,m+t)}{z-m-t+k\choose k}+\frac{1}{d(I^c,m+t)}u(z)\right|\le \\
  \left|{z-m-t+m+t-1\choose m+t-1}\right|+|u(z)|\\
  <\frac{t+m}{t+m+1}\left|{z\choose m+t}\right| + \frac{d(I^c,m)}{d(I^c,m+t)} \left(\frac{2m+1}{t+m+1}\right)\left|{z\choose m+t}\right|\\
  =\left(\frac{t+m}{t+m+1} + \frac{d(I^c,m)}{d(I^c,m+t)} \left(\frac{2m+1}{t+m+1}\right)\right)\left|{z\choose m+t}\right|
\end{align*}

We claim that 
\[
\frac{t+m}{t+m+1} + \frac{d(I^c,m)}{d(I^c,m+t)} \left(\frac{2m+1}{t+m+1}\right)\le 1
\]
equivalently
\begin{eqnarray}\label{eq:more_room}
  d(I^c,m)(2m+1)\le d(I^c,m+t),
\end{eqnarray}
but this is exactly the statement of Lemma~\ref{lem:c}. Therefore we get that
\[
  \left|{z\choose m+t}\right|< \left(\frac{t+m}{t+m+1} + \frac{d(I^c,m)}{d(I^c,m+t)} \left(\frac{2m+1}{t+m+1}\right)\right)\left|{z\choose m+t}\right| \le \left|{z\choose m+t}\right|,
\]
and that is a contradiction. So we obtained that any root of $p(n)$ has length at most $m+t+1$. Therefore if 
\[
  0=d(I^t,z_0)=d(m+t-(m+t-z_0))=(-1)^{m+t}p(m+t-z_0),
\]
then $|m+t-z_0|\le m+t+1$
\end{proof}

\begin{Rem}
With some easy calculation one could get the smallest value $m_0(t)$, for each $t$, such that the conditions of the corresponding proposition is satisfied for any $m\ge m_0(t)$. Specifically it means that if $\max(I)>10$, then one of the conditions are satisfied. For $\max(I)\le 10$ we refer to Figure~\ref{fig:small_cases}, where we included all the possible roots of $d(I,n)$, depending on $m=\max(I)$ and regions ball (blue) of radius $m$ around $0$, ball (blue) of radius $m+1$ around $m$ and ball (red) of radius $(m+1)/2$ around $(m-1)/2$.

Observe that in Proposition~\ref{pro:c_root_big} the crucial inequality was \eqref{eq:more_room}, and checking this condition for the these 84 cases we end up with 16 cases when \eqref{eq:more_room} is not satisfied. 
%
%

\begin{table}[h]
\begin{tabular}{c|cccc}
 $t$ & Corollary~\ref{pro:c} & Condition~\eqref{eq:cond_small} & Condition~\eqref{eq:cond_big} & Condition~\eqref{eq:more_room}\\
 \hline
 0 & \bf{1} & - & - & -\\
 1 & - & \bf{3} & - & -\\
 2 & - & \bf{6} & - & -\\
 3 & - & 14 & \bf{8}& (3) \\
 4 & - & 53 & \bf{3}& (2) \\
 $\ge 5$ & - & - & \bf{1} & (1)
\end{tabular}
\caption{Smallest values for $m_0(t)$, such that the corresponding conditions are satisfied for any $m\ge m_0(t)$. There are 84 $I$'s, that do not satisfy any of the first 3 conditions, and there are 16 of them, that do not satisfy any of the 4 conditions.}
\label{table:values}
\end{table}
%

\end{Rem}

By combining the previous four propositions and checking the uncovered cases of the table (see Figure~\ref{fig:small_cases}) we obtaine the following theorem.
\begin{Th}\label{th:root_main}
 For any $\emptyset\neq I$ if $d(I,z_0)=0$, then
 \begin{enumerate}
  \item $|z_0|\le m$
  \item $|m-z_0|\le m+1$
 \end{enumerate}
 In particular, $-1\le \Re z_0\le m$
\end{Th}

\begin{figure}[h]
 \centering
 \foreach \i in {3,...,10}{
  \begin{subfigure}[p]{0.3\textwidth}
   \includegraphics[width=5cm]{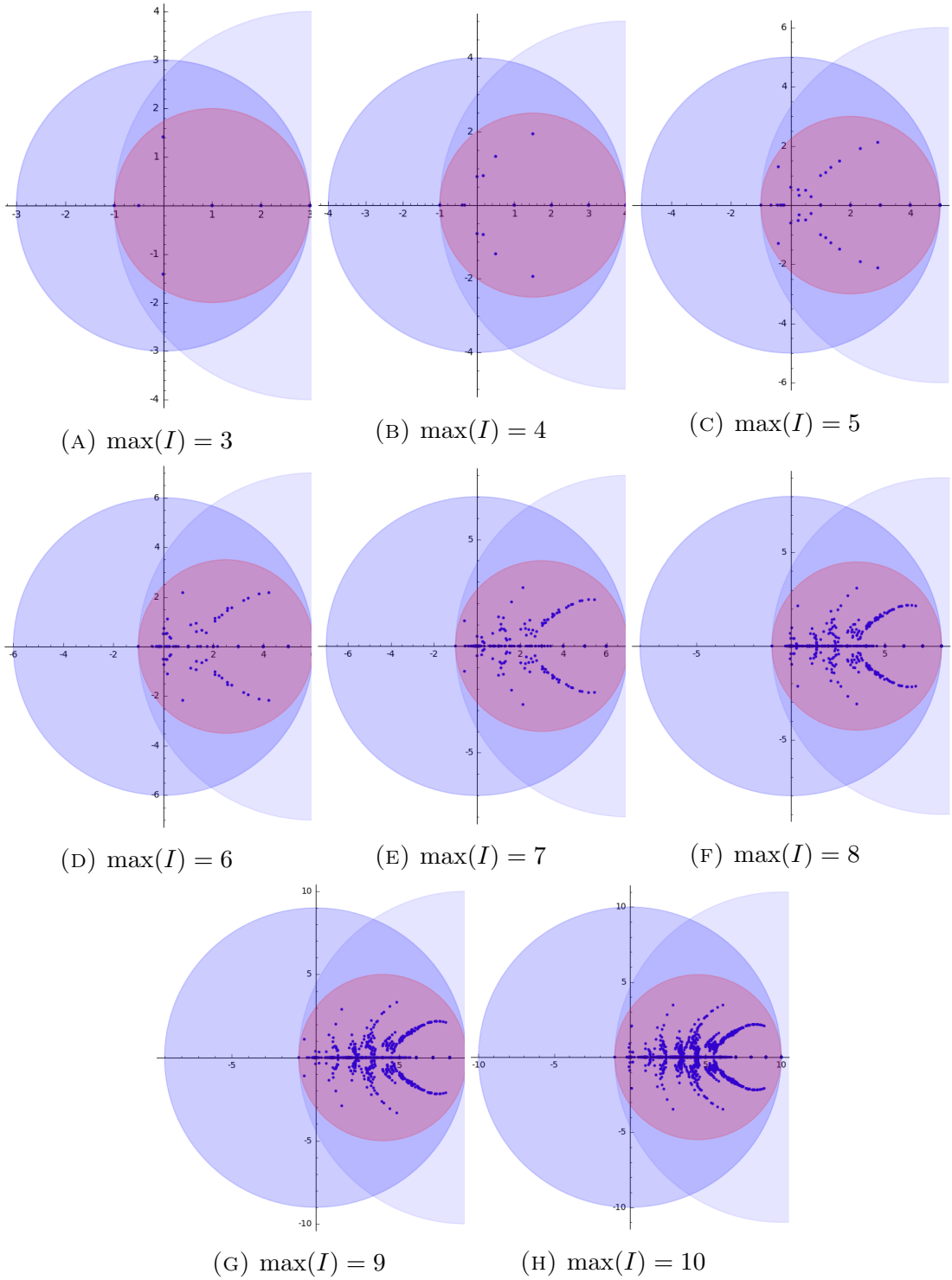}
   \subcaption{$\max(I)=\i$}
  \end{subfigure}
 }
 \caption{Roots of $d(I,n)$ for $m=\max(I)\in\{3,\dots,10\}$ and regions: ball (blue) of radius $m$ around $0$, ball (blue) of radius $m+1$ around $m$ and ball (red) of radius $(m+1)/2$ around $(m-1)/2$}
 \label{fig:small_cases}
\end{figure}




As the previous theorem shows, all the complex roots of $d(I^t,n)$ have their real parts in between -1 and $m+t$. In the following proposition we will show that if $t$ is large enough, then all the roots of $d(I^t,n)$ are real.

\begin{Pro}\label{pro:real}
 Let $I\neq \emptyset $, such that $m-1\notin I$. Then there exists a $t_0=t_0(I)\in \mathbb{N}$, such that for any $t>t_0$ and $v\in\{-1,0,\dots, m+t\}\setminus\{m-1\}$ there exists a unique root of $d(I^t,n)$ of distance 1/4 from $v$. In particular the roots of $d(I^t,n)$ are contained in the interval $[-1,m+t]$.
\end{Pro}
\begin{proof}
 The proof is based on Neumaier's Gershgorin type results on the location of roots of polynomials. For further reference see \cite{Neumaier}. Let \[p_t(n)=\frac{d(I^t,n)}{\prod_{i=1}^t(n-(m+i))}\] and \[T(n)=n(n-1)\dots(n-m+2)(n-m),\] and let us fix the value of $t$.

 Then the leading coefficient of $p_t$ is 
  \[
    \frac{d(I^{t-1},m+t)}{(m+t)!},
  \]
  it has degree $m$, and for $v=0,\dots,m-2,m$ 
  \[
    |\alpha_v|=\frac{|d(I,v)|(m-1-v)}{v!(m-v)!\prod_{i=1}^t(m-v+i)}=\frac{|d(I,v)|(m-1-v)}{v!(m+t-v)!}.
  \]
  Therefore
  \begin{gather*}
  |r_v|=\frac{m}{2}\frac{|d(I,v)|(m-1-v)}{v!(m+t-v)!}\frac{(m+t)!}{d(I^{t-1},m+t)}=\frac{m}{2}\frac{|d(I,v)|(m-1-v)}{d(I^{t-1},m+t)}{m+t\choose v}.
  \end{gather*}
  
  If we are able to prove that $|r_v|\to 0$ as $t\to \infty$ for any $v=0,\dots,m-2,m$, then we would be done.
  
  In order to prove that we observe that
\[
  d(I^{t-1},m+t)\ge d(I^-,m-1){m+t-1\choose t},
\]
since the set of permutations of $D(I^{t-1},m+t)$ with the largest element at position $m$ has size $d(I^-,m-1){m+t-1\choose t}$. To see that, 
choose the largest element $m+t$ into the $m$th position, and take an arbitrary subset of $\{1,\dots,m+t-1\}$ after the $m$th position in a decreasing order, and take the rest as $D(I^-,m-1)$ on the first $m-1$ position through an order-preserving  bijection of the base-set.

Therefore
\begin{gather*}
  |r_v|\le \frac{m(m-1-v)}{2}\frac{|d(I,v)|}{d(I^-,m-1)}\frac{{m+t\choose v}}{{m+t-1\choose t}}=\\
  \frac{m(m-1-v)}{2}\frac{|d(I,v)|}{d(I^-,m-1)}\frac{(m+t)(m-1)!}{v!}\frac{t!}{(m+t-v)!}=\\
  C_{v,m}\frac{(m+t)t!}{(t+m-v)!}.
\end{gather*}

If $v=m$, then $|r_v|=0$, since $d(I,m)=0$.

If $v\in\{0,\dots,m-2\}$, then 
\[
  |r_v|\le C_{v,m}\frac{a_{v,m}(t)}{b_{v,m}(t)},
\]
where $a_{v,m}(t)=t+m$ is a polynomial of degree 1, and $b_{v,m}(t)=\prod_{i=1}^{(m-v)}(t+i)$ is a polynomial of degree at least 2. Therefore $C_{v,m}\frac{a_v(t)}{b_v(t)}\to 0$ as $t\to \infty$, i.e. $|r_v|\to 0$.

\end{proof}

\section{Some remarks and further directions}

We described an interesting phenomenon in Section~\ref{sec:c_base}, namely that $c_k(I)$  and $(-1)^md(I,-n)$ are non-negative integers. This result suggests that there might be some combinatorial proofs for them. 
\begin{Q}
 What do the coefficients $c_k(I)$ and evaluations $(-1)^md(I,-n)$ count?
\end{Q}

There are two conjectures about the roots of the descent polynomial:
\begin{Pro}[Conjecture 4.3. of \cite{Sagan}]
  If $z_0$ is a root of $d(I,n)$, then 
  \begin{itemize}
   \item $|z_0|\le m$,
   \item $\Re z_0\ge-1$.  
  \end{itemize}
\end{Pro}

This conjecture can be viewed as a special case of Theorem~\ref{th:root_main}. 
As a common generalization of the two parts we conjecture that (motivated by numerical computations for $m\le 13$ (e.g. see red regions on  Figure~\ref{fig:small_cases}), by a proof for the case $|I|=1$ and by Proposition~\ref{pro:real}) the roots of $d(I,m)$ will be in a disk with the endpoints of one of its diameters being $-1$ and $m$. More precisely:


\begin{Conj}
 If $d(I,z_0)=0$, then $|z_0-\frac{m-1}{2}|\le\frac{m+1}{2}$.
\end{Conj}

Similarly to the descent polynomial, instead of counting permutations with described descent set, one could ask for the number of permutations with described positions of peaks (i.e. $\pi_{i-1}<\pi_i>\pi_{i+1}$). As it turns out, this peak-counting function is not a polynomial. However, it can be written as a product of a polynomial and an exponential function in a ``natural way''. (See the precise definition in \cite{Sagan13peak}). This polynomial is the so-called peak polynomial. This polynomial behaves quite similarly to the descent one, thus it is natural to ask whether there is a deeper connection between them, or whether we can prove similar propositions to the already obtained ones. In line with this we propose a conjecture about the coefficients in a base similar to $\bar{a}_k(I)$.
\begin{Conj}
 For the peak-polynomial the coefficients in base $\{{n-m+k\choose k+1}\}_{k\in \mathbb{N}}$ form a symmetric, log-concave sequence of non-negative integers.
\end{Conj}

\vspace{10pt}\noindent \textbf{Acknowledgments.}
 I would like to express my sincere gratitude to Bruce Sagan, who 
 pointed out some corollaries of the behavior of  different coefficient sequences. 
 I would also like to thank Alexander Diaz-Lopez 
 for his helpful remarks.
 The research was partially supported by the MTA R\'enyi Institute Lend\"ulet Limits of Structures Research Group.

\bibliography{hivatkozat}
\bibliographystyle{plain}

\end{document}